\documentclass[10pt]{article}
\usepackage{color}
\usepackage{amsmath}
\usepackage{graphicx}
\usepackage{footnote}
\usepackage{subscript}
\usepackage{amsfonts,color}
\usepackage{amsmath,amssymb}
\usepackage{epsfig}
\usepackage{amssymb} 
\usepackage{mathtools}
\usepackage{amsthm,wasysym}
\usepackage{framed}
\usepackage{cancel}

\usepackage[english]{babel}
\usepackage[utf8]{inputenc}
\makeatletter
\usepackage{float}
\usepackage{wrapfig}

\makeatletter
\let\@fnsymbol\@arabic
\makeatother

\usepackage{bbm}
\newcommand{\id}{{\boldsymbol{\mathbbm{1}}}}

\DeclareMathOperator{\sym}{sym}
\DeclareMathOperator{\skyw}{skew}
\DeclareMathOperator{\curl}{curl}
\DeclareMathOperator{\Curl}{Curl}
\DeclareMathOperator{\dyw}{div}
\DeclareMathOperator{\Dyw}{Div}

\DeclareMathOperator{\tr}{tr}

\newcommand{\nn}{\nonumber}

\newcommand{\dhk}{D^h_k}

\newcommand{\R}{{\mathbb R}}

\newcommand{\E}{{\cal E}}

\newcommand{\di}{{\mathrm d}}

\newcommand{\bigchi}{\raisebox{3pt}{$\chi$}}

\renewcommand{\leq}{\leqslant}
\renewcommand{\geq}{\geqslant}
\makeatletter
\let\@fnsymbol\@arabic
\makeatother
\title{{\bf A note on local higher regularity in the dynamic linear relaxed micromorphic model}}
\author{Sebastian Owczarek\,\thanks{Sebastian Owczarek, \ \  Faculty of Mathematics and Information Science, Warsaw University of Technology, ul. Koszykowa 75, 00-662 Warsaw, Poland; email: s.owczarek@mini.pw.edu.pl}\quad
	and \quad	Ionel-Dumitrel Ghiba\footnote{Ionel-Dumitrel Ghiba, Corresponding author:\   Alexandru Ioan Cuza University of Ia\c si, Department of Mathematics,  Blvd. Carol I, no. 11, 700506 Ia\c si,
		Romania;  Octav Mayer Institute of Mathematics of the
		Romanian Academy, Ia\c si Branch,  700505 Ia\c si; email: dumitrel.ghiba@uaic.ro} 	\quad
	and \quad
		Patrizio Neff\,\thanks{Patrizio Neff, \ \  Head of Lehrstuhl f\"{u}r Nichtlineare Analysis und Modellierung, Fakult\"{a}t f\"{u}r Mathematik, Universit\"{a}t Duisburg-Essen, Campus Essen, Thea-Leymann Str. 9, 45127 Essen, Germany, email: patrizio.neff@uni-due.de}}
\date{}

\renewcommand{\theequation}{\thesection.\arabic{equation}}
\newtheorem{tw}{Theorem}[section]

\newtheorem{de}[tw]{Definition}
\newtheorem{col}[tw]{Corollary}
\newtheorem{remark}[tw]{Remark}

\begin{document}
\maketitle
\bigskip
\begin{abstract}

We consider the regularity question of solutions for the dynamic initial-boundary value problem for the linear relaxed micromorphic model. This generalized continuum model couples a wave-type equation for the displacement with a generalized Maxwell-type wave equation for the micro-distortion. Naturally solutions  are found in ${\rm H}^1$ for the displacement $u$ and ${\rm H}(\Curl)$ for the microdistortion $P$.  Using energy estimates for difference quotients, we improve this regularity. We show 
 ${\rm H}^1_{\rm loc}$--regularity for the displacement field, ${\rm H}^1_{\rm loc}$--regularity for the micro-distortion tensor $P$ and  that   ${\rm Curl}\,P$  is ${\rm H}^1$--regular if the data is sufficiently smooth.

\bigskip

\smallskip

\end{abstract}

\section{Introduction}

Generalized continuum theories like the micromorphic or Cosserat model have a long history \cite{Mindlin64,Eringen99}. These models are endowed with additional degrees of freedom (as compared to standard linear elasticity) which are meant to capture effects from a microscale on a continuum level. In the micromorphic family, each macroscopic material point is displaced with the classical displacement $u:\Omega\times[0,T]\rightarrow\mathbb{R}^3$ and attached to the macroscopic point there is an affine micro-distortion field $P:\Omega\times[0,T]\rightarrow\mathbb{R}^{3\times3}$  describing the interaction with the microstructure. 

The equations of motion are obtained from Hamilton's principle and the introduction of suitable kinetic and elastic energy expressios. Typical for the classical micromorphic model is a quadratic term  $\frac{1}{2}\lVert \nabla P\rVert^2$ leading to an equilibrium problem of the type
\begin{eqnarray}
\label{1}
u_{,tt}&=&\Dyw\big(\nabla u-P\big)+f=\Delta u-\Dyw P+f,\\[1ex]
P_{,tt}&=&(\nabla u-P)- \sym P+\Delta P+M.\notag
\end{eqnarray}
Unique  solutions are found in classical Hilbert spaces, $u\in {\rm C}([0,T];{\rm H}
^1(\Omega, \mathbb{R}^3))$ and $P\in {\rm C}([0,T];{\rm H}
^1(\Omega, \mathbb{R}^{3\times 3}))$ for all times $T>0$, under suitable initial and boundary conditions and for these models \cite{picard,IesanNappa2001}, higher regularity follows from the standard approach for the wave equation. 

Departing from this framework, Neff and his co-workers have introduced the so called relaxed micromorphic model \cite{NeffGhibaMicroModel}. Here, the ``curvature energy'' $\frac{1}{2}\lVert \nabla P\rVert^2$ is replaced by $\frac{1}{2}\lVert \Curl P\rVert^2$, which represents a sort of ``relaxation'' since the interaction-strength of the microstructure with itself is getting much weaker. The typical set of equations turns into 
\begin{align}
\label{2}
u_{,tt}&=\Dyw\big(\nabla u-P\big)+f,\\[1ex]
P_{,tt}&=(\nabla u-P)- \sym P+\Curl \Curl P+M,\notag
\end{align}
under suitable initial and boundary conditions, which for $P$ has to be a tangential boundary condition. The second equation can be seen as a generalized Maxwell-problem due to the $\Curl \Curl$-operator. Unique solutions are now found in the natural setting $u\in {\rm C}([0,T];{\rm H})
^1(\Omega, \mathbb{R}^3)$ and $P\in {\rm C}([0,T];{\rm H}
(\Curl; \Omega))$ for all times $T>0$, see \cite{GhibaNeffExistence,Sebastian1}.

The current interest for the relaxed micromorphic model mainly stems from the fact that it is able to describe frequency-band gaps as observed in real meta-materials  (see e.g. \cite{blanco2000large,liu2000locally} and \cite{MadeoNeffGhibaW,madeo2016reflection,d2017panorama,aivaliotis2019microstructure,barbagallo2019relaxed,aivaliotisfrequency,d2019effective}). 
When it comes to the numerical implementation, it would be natural to use the pair of function spaces  ${\rm H}
^1(\Omega, \mathbb{R}^3)\times {\rm H}
(\Curl; \Omega)$. If, on the other hand, a standard ${\rm H}
^1(\Omega, \mathbb{R}^3)\times {\rm H}
^1(\Omega, \mathbb{R}^{3\times 3})$ framework is used (which is more easily available from relevant software packages) the occurring approximation error hinges on the local regularity of solutions. 
Indeed, in this paper we show (Theorem \ref{mr}) that $u\in {\rm H}^2_{\rm loc}(\Omega)$ and $P\in {\rm H}^1_{\rm loc}(\Omega)$ can be achieved under suitable regularity assumptions on the data. This establishes convergence of standard numerical FEM-implementations in the interior and constitutes a major motivation for our work. While in the static case, this kind of improved regularity (from $H(\Curl;\Omega)$ to ${\rm H}^1_{\rm loc}(\Omega)$) for the microdistortion field $P$ is illusory, the dynamic formulation provides much more control of the appearing fields through the kinematic terms.

In a previous paper \cite{Sebastian1}, starting from some results established by Alonsi and Valli \cite{tantrace}, we have seen that for all $G\in \widetilde{\chi}_{\partial\Omega}:= \{G\in {\rm H}^{\frac{1}{2}}(\partial\Omega)\,\,;\,\, \bigl\langle G_i\big|_{\partial\Omega}, n\bigr\rangle=0\}$ there exists an  extension $\widetilde{G}\in {\rm H}(\Curl;\Omega)$ such that $\Curl \Curl \widetilde{G}=0$ and which belongs actually  to ${\rm H}^1(\Omega)$. This result is useful in order to prove that the initial-boundary value problem with non-homogeneous boundary condition admits a unique solution $(u, P)$ with the regularity: 
$
u\in {\rm C}^1([0,T];{\rm H}^1(\Omega ))\,, u_{,tt}\in {\rm C}([0,T];{\rm L}^2(\Omega ))\,,$ $
P\in {\rm C}^1([0,T]; {\rm H}(\Curl;\Omega )),\  P_{,tt}\in {\rm C}([0,T];{\rm L}^2(\Omega)), $ $  \Curl\Curl P\in {\rm C}([0,T];{\rm L}^2(\Omega  )),$ for all times $T>0$.
Moreover, we have shown that any extension
 $\widetilde{v}\in{\rm H}(\curl;\Omega)$ of $v\in \tilde{\bigchi}_{\partial\Omega}$ is such that $\nabla\curl \tilde {v}\in {\rm L}^2(\Omega)$, see \cite{Sebastian1}. In the present paper, starting from this remark and by using some standard techniques, we prove  that under suitable assumptions on the data, the solution $(u, P)$ is in fact smoother, i.e.
$
 u\in {\rm L}^{\infty}(0,T;{\rm H}^2_{\rm loc}(\Omega))\,,  P\in {\rm L}^{\infty}(0,T;{\rm H}^1_{\rm loc}(\Omega))   \mathrm{\,\,and} \
 \Curl P\in {\rm L}^{\infty}(0,T;{\rm H}^1(\Omega))\,,
 $
 for all times $T>0$. We point out that this result may seem surprisingly  unusual, since we have no information on ${\rm Div} \,P$ in $\Omega$ and not even on $P_i\cdot n$ ($i=1,2,3$) on $\partial \Omega$.

\section{The relaxed micromorphic model}
\renewcommand{\theequation}{\thesection.\arabic{equation}}
\setcounter{equation}{0}%

We consider $\Omega$ to be a connected, bounded, open subset of $\mathbb{R}^3$ with a  ${\rm C}^{1,1}$  boundary $\partial\Omega$ and $T > 0$ is a fixed length of the time interval. The domain $\Omega$ is occupied by a micromorphic continuum whose motion  is referred to a fixed system of rectangular Cartesian axes $Ox_i$ $(i=1,2,3)$. 
\subsection{Notations}
 Throughout this paper {(if we do not specify otherwise)} Latin subscripts take the values $1,2,3$. 
We denote by $\mathbb{R}^{3\times 3}$ the set of real $3\times 3$ matrices.   For all $X\in\mathbb{R}^{3\times3}$ we set ${\rm sym}\, X=\frac{1}{2}(X^T+X)$ and ${\rm skew} X=\frac{1}{2}(X-X^T)$.
The standard Euclidean scalar product on $\mathbb{R}^{3\times 3}$ is given by
$\langle {X},{Y}\rangle_{\mathbb{R}^{3\times3}}=\tr({X Y^T})$, and thus the Frobenius tensor norm is
$ \lVert {X}\rVert^2=\langle{X},{X}\rangle_{\mathbb{R}^{3\times3}}$. In the following we omit the index
$\mathbb{R}^{3\times3}$. The identity tensor on $\mathbb{R}^{3\times3}$ will be denoted by $\id$, so that
$\tr({X})=\langle{X},{\id}\rangle$. Typical conventions for differential
operations are implied such as comma followed
by a subscript to denote the partial derivative with respect to
the corresponding Cartesian coordinate, while $t$ after a comma denotes the partial derivative with respect to the time. A matrix having the  three  column vectors $A_1,A_2, A_3$ will be written as 
$
(A_1\,|\, A_2\,|\,A_3).
$

We denote by $u:\Omega\times [0,T]\rightarrow \R^3$  the displacement vector of the material point, while $P:\Omega\times [0,T]\rightarrow \R^{3\times 3}$ describes the substructure of the material which can rotate, stretch, shear and shrink (the micro-distortion). Here, $T > 0$ is a fixed length of the time interval.  
For vector fields $u=\left(    u_1, u_2,u_3\right)$ with  $u_i\in {\rm H}^{1}(\Omega)$, $i=1,2,3$,
we define
$
\nabla \,u:=\left(
\nabla\,  u_1\,|\,
\nabla\, u_2\,|\,
\nabla\, u_3
\right)^T,
$
while for tensor fields $P$ with rows in ${\rm H}({\rm curl}\,; \Omega)$, i.e.
$
P=\begin{pmatrix}
(P^T.e_1)^T\,|\,
(P^T.e_2)^T\,|\,
(P^T.e_3)^T
\end{pmatrix}^T$, $P^T.e_i\in {\rm H}({\rm curl}\,; \Omega)$, $i=1,2,3$,
we define
$$
{\rm Curl}\,P:=\begin{pmatrix}
{\rm curl}\, (P^T.e_1)^T\,|\,
{\rm curl}\, (P^T.e_2)^T\,|\,
{\rm curl}\, (P^T.e_3)^T
\end{pmatrix}^T
.
$$
 The corresponding Sobolev spaces for the second order tensor fields $P$, ${\rm Curl}\, P$ and $\nabla\, u$  will be denoted by
$
{\rm H}^1(\Omega) \ \ \text{and}\  \ {\rm H}({\rm Curl}\,; \Omega)\, ,
$ and $
{\rm H}^1_0(\Omega) \ \ \text{and}\  \ {\rm H}_0({\rm Curl}\,; \Omega)\, ,
$ respectively.

\subsection{The initial-boundary value problem in the linear relaxed micromorphic theory}\label{sectmodel}

The partial differential equations associated to the dynamical relaxed  micromorphic model \cite{NeffGhibaMicroModel} are
\begin{eqnarray}
\label{1.2}
u_{,tt}&=&\Dyw\big(2\,\mu_{\rm e}\sym(\nabla u-P)+2\,\mu_{\rm c}\skyw(\nabla u-P)+\lambda_{\rm e}\tr(\nabla u-P)\id\big)+f\nn\\[1ex]
P_{,tt}&=&2\,\mu_{\rm e}\sym(\nabla u-P)+2\,\mu_{\rm c}\skyw(\nabla u-P)+\lambda_{\rm e}\tr(\nabla u-P)\id\nn\\[1ex]
&&-(2\,\mu_{\mathrm{micro}} \sym P+\lambda_{\mathrm{micro}}(\tr P)\id)-\mu_{\rm micro}\, L_{\rm c}^2\,\Curl\Curl P+M\,,
\end{eqnarray}
in $\Omega\times (0,T)$, where $f:\Omega\times (0,T)\rightarrow \R^3$ is a given body force and $M:\Omega\times (0,T)\rightarrow \R^{3\times 3}$ is a given body moment tensor.

Here, the constants  $\mu_{\rm e},\lambda_{\rm e},\mu_{\rm c}, \mu_{\rm micro}, \lambda_{\rm micro}$  are constitutive parameters describing the isotropic elastic response of the material, while $L_{\rm c}>0$ is the characteristic length of the relaxed micromorphic
model. We assume that the constitutive parameters are such that 
\begin{align}\label{condpara}
\mu_{\rm e}>0,\quad\quad  2\,\mu_{\rm e}+3\,\lambda _{\rm e}>0,\quad\quad  \mu_{\rm c}\geq 0,\quad\quad  \mu_{\rm micro}>0, \quad\quad  2\,\mu_{\rm micro}+3\,\lambda _{\rm micro}>0.
\end{align}

The system \eqref{1.2} is considered with the boundary conditions 
\begin{equation}
\begin{split}
u(x,t)=g(x,t),\quad \quad P_i(x,t)\times n(x)=G_i(x,t)
\end{split}
\label{1.3}
\end{equation}
for $(x,t)\in \partial\Omega\times [0,T]$, where $n$ is the unit normal vector at the surface $\partial\Omega$, $\times$ denotes the vector product and $P_i$ ($i=1,2,3$) are the rows of $P$. The model is also driven by the following initial conditions 
\begin{equation}
\begin{split}
u(x,0)=u^{(0)}(x)\,,\quad u_{,t}(x,0)=u^{(1)}(x)\,,\quad P(x,0)=P^{(0)}(x)\,,\quad P_{,t}(x,0)=P^{(1)}(x)
\end{split}
\label{1.4}
\end{equation}
for $x\in\Omega$.

\begin{de}
We say that the initial data $(u^{(0)},u^{(1)},P^{(0)},P^{(1)})$ satisfy the compatibility condition if
\begin{equation}
u^{(0)}(x)=g(x,0)\,,\quad \ \,  
u^{(1)}(x)=g_{,t}(x,0)\,, \quad 
P^{(0)}_{i}(x)\times n(x)=G_i(x,0)\,,\quad 
P^{(1)}_{i}(x)\times n(x)=G_{i,t}(x,0)
\label{comcon}
\end{equation}
for $x\in\partial\Omega$ and $i=1,2,3$, where $G_{i,t}$ denotes the time derivative of the function $G_i$.
\end{de}

\subsection{Preliminary results}

In a previous paper \cite{Sebastian1} we have considered the space 
\begin{equation}
\tilde{\bigchi}_{\partial\Omega}:=\{v\in {\rm H}^{\frac{1}{2}}(\partial\Omega)\mid\bigl\langle v\big|_{\partial\Omega}, n\bigr\rangle=0\}
.
\end{equation}
This space is related to the fact that, according to \cite{tantrace} and  \cite[p. 34]{Giraultbook}, for all $v\in {\rm H}(\curl;\Omega)$  the tangential trace $n\times v\big|_{\partial \Omega}$  belongs to  a proper subspace of 
${\rm H}^{-\frac{1}{2}}(\partial\Omega)$ defined  by 
\begin{equation}
\bigchi_{\partial\Omega}:=\{v\in {\rm H}^{-\frac{1}{2}}(\partial\Omega)\mid\bigl\langle v\big|_{\partial\Omega}, n\bigr\rangle=0\,\,\mathrm{and}\,\,\dyw_{\tau} v\in {\rm H}^{-\frac{1}{2}}(\partial\Omega)\}
\end{equation}
and equipped with the norm 
\begin{equation}
 \lVert v \rVert_ {\bigchi_{\partial\Omega}}= \lVert v \rVert_ {{\rm H}^{-\frac{1}{2}}(\partial\Omega)}+ \lVert\dyw_{\tau} v \rVert_ {{\rm H}^{-\frac{1}{2}}(\partial\Omega)}\,,
\end{equation}
where $\dyw_{\tau} v$ is the tangential divergence of the vector $v$. Let us recall that the tangential divergence $\dyw_{\tau} v$ of the vector $v$ is the distribution in ${\rm H}^{-\frac{3}{2}}(\partial\Omega)$ which satisfies
	
	\begin{equation*}
	[[\dyw_{\tau} v,w]]_{\partial\Omega}=-\big[ v,(\nabla w^{\ast})\big|_{\partial\Omega}\big]_{\partial\Omega}\quad \forall\quad w\in {\rm H}^{\frac{3}{2}}(\partial\Omega)\,,
	\label{2.3}
	\end{equation*}
	where $w^{\ast}\in {\rm H}^{2}(\Omega)$ is any extension of $w$ in $\Omega$. Here, $\big[\cdot,\cdot\big]$ denotes the duality pair between the space ${\rm H}^{-\frac{1}{2}}(\partial\Omega)$ and ${\rm H}^{\frac{1}{2}}(\partial\Omega)$.  $[[\cdot,\cdot]]$ denotes the duality pair between the space ${\rm H}^{-\frac{3}{2}}(\partial\Omega)$ and ${\rm H}^{\frac{3}{2}}(\partial\Omega)$ (more information can be found in \cite{tandyw}).

We observe that $\tilde{\bigchi}_{\partial\Omega}\subset \bigchi_{\partial\Omega}$. In \cite{Sebastian1}, summarising some results presented in  \cite{tantrace} and \cite[Theorem 6 of Section 2]{electrobook}, we have concluded that the following results hold true.
\begin{tw}
	\label{lem:2.2}
	Assume that the boundary $\partial\Omega$ is of class ${\rm C}^{1,1}$  or that $\Omega$ is a convex polyhedron. Moreover, let us assume that $v\in {\bigchi}_{\partial\Omega}$. Then there exists an extension $\widetilde {v}\in {\rm H}(\curl;\Omega)$ of $v$ in $\Omega$ such that
		\begin{enumerate}
			\item[1.] $\curl\curl\widetilde {v}=0\,;$ \qquad \qquad 
			2. $\dyw\,\widetilde {v}=0$\,;
			\qquad \qquad  3.  $\widetilde {v}\in {\rm H}^1(\Omega)$ \ \ for all\ \  $v\in \widetilde{\chi}_{\partial\Omega}$.
		\end{enumerate}

\end{tw}

\begin{col} 	\label{collem}
The construction of the extension operator defined in the proof of Theorem \ref{lem:2.2} yields that $\nabla\curl \widetilde {v}\in {\rm L}^2(\Omega)$ (we refer to \cite{Sebastian1} for more details).
\end{col} 

Using these results, we have proven the existence and uniqueness of the solution of the initial-boundary value problem arising in the linear relaxed theory for non-homogeneous boundary conditions, see \cite{Sebastian1}. 
\begin{tw} {\rm (Existence of solution with non-homogeneous boundary conditions)}\label{existenceresult}
	Let us assume that the constitutive parameters satisfy  \eqref{condpara} and the initial data are such that
\begin{equation}
(u^{(0)}, u^{(1)}, P^{(0)}, P^{(1)})\in {\rm H}^1(\Omega; \R^3 )\times {\rm H}^1(\Omega; \R^3 )\times {\rm H}(\Curl;\Omega )\times {\rm H}(\Curl;\Omega )\,
\label{2.27}
\end{equation}
and that the compatibility condition \eqref{comcon} holds.
Additionally,
\begin{equation}
\begin{split}
\Dyw\big(2\,\mu_{\rm e}\sym(\nabla u^{(0)}-P^{(0)})+2\,\mu_{\rm c}\skyw(\nabla u^{(0)}-P^{(0)})+\lambda_{\rm e}\tr(\nabla u^{(0)}-P^{(0)})\id\big)&\in {\rm L}^2(\Omega)\,,
\\
\Curl\Curl P^{(0)}&\in {\rm L}^2(\Omega)
\end{split}
\end{equation}
and
$
f\in {\rm C}^1([0,T];{\rm L}^2(\Omega))\,,\  M\in {\rm C}^1([0,T];{\rm L}^2(\Omega))$, $
g\in {\rm C}^3([0,T];{\rm H}^{\frac{3}{2}}(\partial\Omega))\,,\  G_i\in {\rm C}^3([0,T];\tilde{\chi}_{\partial\Omega})\,,\,\, i=1,2,3\,.
$
Then, the system \eqref{1.2} with boundary conditions \eqref{1.3} and initial conditions \eqref{1.4} possesses a global in time, unique solution $(u, P)$ with the regularity: for all times $T>0$ 
\begin{equation}
\begin{split}
u\in {\rm C}^1([0,T];{\rm H}^1(\Omega ))\,,&\  u_{,tt}\in {\rm C}([0,T];{\rm L}^2(\Omega ))\,,\ 
P\in {\rm C}^1([0,T]; {\rm H}(\Curl;\Omega )),\  P_{,tt}\in {\rm C}([0,T];{\rm L}^2(\Omega))
\end{split}
\label{2.32}
\end{equation}
Moreover,
\begin{equation}
\Dyw\big(2\,\mu_{\rm e}\sym(\nabla u-P)+2\,\mu_{\rm c}\skyw(\nabla u-P)+\lambda_{\rm e}\tr(\nabla u-P)\id\big)\in {\rm C}([0,T];{\rm L}^2(\Omega  ))\\[1ex]
\label{2.33}
\end{equation}
and 
\begin{equation}
\Curl\Curl P\in {\rm C}([0,T];{\rm L}^2(\Omega  ))\,.
\label{2.34}
\end{equation}
\label{tw:nonhomo}
\end{tw}
\begin{col} 	\label{collem1}
Theorem \ref{tw:nonhomo} assumes that $G_i\in {\rm C}^3([0,T];\tilde{\chi}_{\partial\Omega})$ for $i=1,2,3$, therefore the Theorem \ref{lem:2.2} and Corollary \ref{collem} yield that there exists an extension $\tilde{G}_i\in {\rm C}^3([0,T]; {\rm H}(\curl;\Omega ))$ such that $n\times \tilde{G}_i=G_i$ on $\partial\Omega$ in the sense of ${\rm H}^{-\frac{1}{2}}(\partial\Omega)$, $\curl\curl  \tilde{G}_i=0$ and $\nabla \tilde{G}_i\in {\rm C}^3([0,T]; {\rm H}^1(\Omega ))$.
\end{col} 
The assumption on the constitutive parameters was used in the proof since we need to know that there exists a constant $C>0$ such that
\begin{align}
{\rm C}( \lVert\nabla u\rVert^2_{{\rm L}^2(\Omega)}+ \lVert P\rVert^2_{{\rm H}(\Curl;\Omega)})\leq \int_{\Omega}\Big(&\mu_{\rm e} \lVert\sym(\nabla u-P)\rVert^2+\mu_{\rm c} \lVert {\rm skew}(\nabla u-P)\rVert^2+\frac{\lambda_{\rm e}}{2}[\tr(\nabla u-P)]^2\\
&+ \mu_{\mathrm{micro}}  \lVert\sym P\rVert^2+\frac{\lambda_{\mathrm{micro}}}{2}\,[\tr (P)]^2+\frac{\mu_{\rm micro} \,L_{\rm c}^2}{2} \,\lVert\Curl P\rVert^2\Big)\,\di x\,\notag
\end{align}
for all $u\in {\rm H}^1_0(\Omega)$ and $P\in{\rm H}_0({\rm Curl}\, ; \Omega)$. 
This coercivity follows even when $\mu_{\rm c}=0$, due to the fact that 
\cite{NeffPaulyWitsch,BNPS2,LMN}
there exists a positive constant $C$, only depending on $\Omega$, such that for all $P\in{\rm H}_0({\rm Curl}\, ; \Omega)$ the following estimate holds
	\begin{align}
	 	{ \lVert  P \rVert_ {{\rm H}(\mathrm{Curl})}^2}:= \lVert  P \rVert_ {{\rm L}^2(\Omega)}^{ {2}}+ \lVert  \Curl P \rVert_ {{\rm L}^2(\Omega)}^{ {2}}&\leq C\,( \lVert  {\rm sym} P\rVert^2_{{\rm L}^2(\Omega)}+ \lVert  \Curl P\rVert^2_{{\rm L}^2(\Omega)}). 
	\end{align}

Let us adjoin the total energy to a solution of the initial-value problem 
\begin{align}\label{1.5}
\E(u, P)(t)=&\frac{1}{2}\int_{\Omega}( \lVert u_{,t}\rVert^2+ \lVert P_{,t}\rVert^2)\,\di x + \int_{\Omega}\Big(\mu_{\rm e} \lVert \sym(\nabla u-P)\rVert^2+\frac{\lambda_{\rm e}}{2}\,[\tr(\nabla u-P)]^2
+\mu_{\rm c}\, \lVert \skyw(\nabla u-P)\rVert^2
\notag\\[1ex]
&\qquad \quad \qquad \quad \qquad \quad\qquad  \  +\mu_{\rm micro} \lVert \sym P\rVert^2+\frac{\lambda_{\mathrm{micro}}}{2}\,[\tr (P)]^2
+\frac{\mu_{\rm micro}\, L_{\rm c}^2}{2}\, \lVert \Curl P\rVert^2\Big)\,\di x\,.
\end{align}

\section{Higher local regularity }
\renewcommand{\theequation}{\thesection.\arabic{equation}}
\setcounter{equation}{0}%
The aim of this section is to show the higher local regularity of the solution of problem \eqref{1.2}. We will use the difference quotient method. Let us consider bounded open subsets of $\Omega$ with smooth boundaries (so we can use Korn's inequality) such that $V\Subset U\Subset\Omega$. We consider  a cutoff function $\eta\,:\R^3\rightarrow [0,1]$ with the following properties
\begin{equation}
\eta\,:\Omega\rightarrow [0,1]\,,\qquad  \eta\,\in C_0^{\infty}(\R^3)\,,\qquad 
\eta\,=1\quad \mathrm{on}\,\, V  \quad \mathrm{and}\quad \eta\,=0\quad \mathrm{on}\,\, \Omega\setminus U\,. 
\label{4.1}
\end{equation}
We will denote by $D^h_k$ the difference quotient in the direction $\vec{e}_k$ with the step $h$ i.e. for any function $\phi$ defined on $\Omega$, any $h\in\R$ sufficiently small ($k=1,2,3$), set
\begin{equation}
\begin{split}
D^h_k\phi(x):=\frac{\phi(x+h\vec{e}_k)-\phi(x)}{h}\,. 
\end{split}
\label{4.2}
\end{equation}
Observe that for $0<|h|<\frac{1}{2}\mathrm{dist}(U,\partial\Omega)$ and $x\in U$ the difference quotient is well-defined. Moreover, the products $\eta\, D^h_k(\cdot)$ are equal to zero for $x\notin U$. Let us recall  Theorem 3 from  Section 5.8.2 of \cite{evansbook} on the relation between the difference quotient and weak derivatives.
\begin{tw}
		\begin{description}
			\item[]
			\item[(i)] Assume that $1\leq p < \infty$ and $\phi\in W^{1,p}(\Omega)$. Then for all 
			$k=1,2,3$ and all $V\Subset U\Subset\Omega$ it holds 
			\begin{equation}
			\begin{split}
			 \lVert D^h_k\phi \rVert_ {L^{p}(V)}\leq C\, \lVert \nabla\phi \rVert_ {L^{p}(U)}
			\end{split}
			\label{difguo1}
			\end{equation}
			for some constant $C={\rm C}(p,U)$ and all $0<|h|<\frac{1}{2}\mathrm{dist}( V,\partial U)$.
			\item[(ii)] In turn, if $1<p<\infty$, $\phi\in L^p(\Omega)$ and there exists constant $C>0$ such that for all $k=1,2,3$ and all $0<|h|<\frac{1}{2}\mathrm{dist}( V,\partial U)$ the following inequality 
			\begin{equation}
			\begin{split}
			 \lVert D^h_k\phi \rVert_ {L^{p}(V)}\leq C
			\end{split}
			\label{difguo2}
			\end{equation}
			is fulfilled. Then, 
			\begin{equation}
			\phi\in W^{1,p}(V)\quad\mathrm{and}\quad  \lVert \nabla\phi \rVert_ {L^{p}(V)}\leq C\,.
			\label{difguo3}
			\end{equation}
		\end{description}
	\label{tw:difquo}
\end{tw}
\noindent
In the proof of the main estimate we will need the following information about tensor $P$
\begin{equation}
\begin{split}
 \lVert \nabla\curl P_i \rVert_ {{\rm C}([0,T];{\rm L}^2(\Omega))}\leq C\,,
\end{split}
\label{ineq}
\end{equation}
where $i=1,2,3$, $
P=\left(\begin{array}{c}
P_1
\,|\,
P_2
\,|\,
P_3
\end{array}\right)^T
$ is the unique solution of the problem \eqref{1.2}-\eqref{1.4} and the constant $C>0$ depends on $\Omega$ ($\Curl P$ is calculated with respect to rows of the matrix $P$). Observe that it is sufficient to prove the inequality \eqref{ineq} with homogeneous tangential boundary condition $P_i\times n =0$ on $\partial\Omega$ since the estimates in the case of inhomogeneous tangential boundary condition  follows as a consequence of Corollary \ref{collem1}.

\noindent
Assuming that $P_i\times n=0$ on $\partial\Omega$ we obtain from Theorem \ref{existenceresult} that $P_i\in {\rm C}([0,T];{\rm H}_0(\curl;\Omega))$, which means, that $P_i\in {\rm C}([0,T];{\rm H}(\curl;\Omega))$ and it has vanishing tangential component on the boundary $\partial\Omega$. Using ${\rm C}_0^{\infty}(\Omega)$- vector fields, one can show by a standard closure procedure that $\curl P_i\in {\rm C}([0,T];{\rm H}_0(\mathrm{div};\Omega))$ and $\mathrm{div}\curl P_i=0$ ($\curl P_i$ has vanishing normal component on the boundary $\partial\Omega$)-see for example \cite[Proposition 6.1.5]{Seifert}. Now the inequality \eqref{ineq} is a consequence of the inequality know in the literature as  Gaffney's inequality: for $v\in {\rm H}(\curl;\Omega)\cap {\rm H}_0(\dyw;\Omega)$ the following inequality  
\begin{equation}
\begin{split}
 \lVert \nabla v \rVert_ {{\rm L}^2(\Omega)}\leq {\rm C}\,( \lVert \curl v \rVert_ {{\rm L}^2(\Omega)}+ \lVert \dyw v \rVert_ {{\rm L}^2(\Omega)}+ \lVert v \rVert_ {{\rm L}^2(\Omega)})\,,
\end{split}
\label{ineq2}
\end{equation}
is satisfied, where the constant $C>0$ does not depend on $v$. Some standard references on \eqref{ineq2} are Amrouche-Bernardi-Dauge-Girault \cite{Dauge}, Costabel \cite{Costabel}, Dautray-Lions \cite{Lions}, and Grisvard \cite[p.~318]{electrobook} (see also \cite{Costabel90} and \cite{Costabel99}). The regularity \eqref{2.34} entails that $\curl P_i\in {\rm H}(\curl;\Omega)\cap {\rm H}_0(\dyw;\Omega)$, hence applying inequality \eqref{ineq2} with $v=\curl P_i$ we obtain the inequality \eqref{ineq}. Now we are ready to prove the main estimate of this article.
\begin{tw} {\rm (Main estimate) } 
	Suppose that  $(u,P)$ is the solution of the problem \eqref{1.2} with $P_i(x,t)\times n(x)=0$ on $\partial \Omega$ and under the hypotheses  of Theorem \ref{existenceresult}. 
	Moreover, assume that the given forces have the regularity
	$
	f\in {\rm L}^2(0,T;{\rm H}^1_{\rm loc}(\Omega))$, $ M\in {\rm L}^2(0,T;{\rm H}^1_{\rm loc}(\Omega))
	$ for all $T>0$
	and the initial data admit the regularity
	$
	\nabla u^{(0)}-P^{(0)}\in {\rm H}^1_{\rm loc}(\Omega)\,,  \sym P^{(0)}\in {\rm H}^1_{\rm loc}(\Omega)\,,$ $ 
	\tr P^{(0)}\in {\rm H}^1_{\rm loc}(\Omega)\,, $ $ \Curl P^{(0)}\in {\rm H}^1_{\rm loc}(\Omega)\,, $ $
	u^{(1)}\in {\rm H}^1_{\rm loc}(\Omega;\R^{3})\,$ $  \mathrm{and}\quad P^{(1)}\in {\rm H}^1_{\rm loc}(\Omega)\,.
	$
	Then, for all $k\in\{1,2,3\}$, $t\in [0,T]$ and sufficiently small $h\in\R$ the following inequality
	\begin{equation}
	\E(\eta\, D^h_k u, \eta\, D^h_k P)(t)\leq C
	\label{estimate}
	\end{equation}
	holds, where $(u,P)$ is the solution of the system \eqref{1.2}, $\eta$ is the function defined by \eqref{4.1} and the constant $C>0$ does not depend on $h$ or on $\eta$.
	\label{tw:regu}
\end{tw}
\begin{proof} From Theorem \ref{existenceresult} we got that a solution $(u,P)$  of the problem \eqref{1.2}  with $P_i(x,t)\times n(x)=0$ on $\partial \Omega$ has the following regularity
	\begin{equation}
	\begin{split}
	u\in {\rm C}^1([0,T];{\rm H}^1(\Omega))\,,&\quad u_{,tt}\in {\rm C}([0,T];{\rm L}^2(\Omega ))\,,\\[1ex]
	P\in {\rm C}^1([0,T]; {\rm H}_0(\Curl;\Omega ))\,, &\quad P_{,tt}\in {\rm C}([0,T];{\rm L}^2(\Omega)),\\[1ex]
	\Curl\Curl P\in {\rm C}([0,T];{\rm L}^2(\Omega  ))\,\quad &\mathrm{and}\quad \nabla\Curl P\in {\rm C}([0,T];{\rm L}^2(\Omega))\,,
	\end{split}
	\label{assu}
	\end{equation}
where the last regularity statement in \eqref{assu} follows from inequality \eqref{ineq}. Fix $k\in\{1,2,3\}$ and assume that $h$ is sufficiently small. Calculating the time derivative of the energy \eqref{1.5} evaluated on localised differences provides
\begin{align}\label{4.3}
\frac{d}{dt}\big(\E(\eta\, D^h_k u, \eta\, D^h_k P)(t)\big)=&\int_{\Omega}\Big[\langle\eta\, D^h_k u_{,t},\eta\, D^h_k u_{,tt}\rangle + \langle\eta\, D^h_k P_{,t},\eta\, D^h_k P_{,tt}\rangle \Big]\,\di x\notag\\[1ex]
&+ \int_{\Omega}\Big[2\,\mu_{\rm e}\langle\eta\,\sym(\nabla D^h_k u-D^h_k P),\eta\,\sym(\nabla \dhk u_{,t}-D^h_k P_{,t})\rangle\notag\\[1ex]
&\qquad \quad +\lambda_{\rm e}\,\eta\,\tr(\nabla\dhk u-D^h_k P)\,\eta\,\tr(\nabla\dhk u_{,t}-D^h_k P_{,t})\\[1ex]
&\qquad \quad +2\,\mu_{\rm c}\langle\eta\,\skyw(\nabla\dhk u-D^h_k P),\eta\,\skyw(\nabla\dhk u_{,t}-D^h_k P_{,t})\rangle\notag\\[1ex]
&\qquad \quad +2\,\mu_{\mathrm{micro}} \,\eta^2\langle\sym \dhk P,\sym\dhk P_{,t}\rangle+\lambda_{\mathrm{micro}}\tr(\dhk P)\,\tr(\dhk P_{,t})\notag\\[1ex]
&\qquad \quad +\mu_{\rm micro} L_{\rm c}^2\langle\eta\,\curl \dhk P, \eta\,\curl \dhk P_{,t}\rangle\Big]\,\di x\notag
\end{align}
\begin{align}
\qquad \qquad\qquad\qquad\quad=&\int_{\Omega}\Big[\langle\eta\, D^h_k u_{,t},\eta\, D^h_k u_{,tt}\rangle+\langle 2\,\mu_{\rm e}\,\eta\,\sym(\nabla\dhk u-D^h_k P)+ \lambda_{\rm e}\,\eta\,\tr(\nabla\dhk u-D^h_k P)\id\notag\\[1ex]
&\qquad \qquad \qquad \qquad\qquad\qquad+ 2\,\mu_{\rm c}\,\eta\,\skyw(\nabla\dhk u-D^h_k P),\eta\,\nabla\dhk u_{,t}\rangle\Big]\,\di x\notag\\[1ex]
&+\int_{\Omega}\Big[\langle\eta\, D^h_k P_{,t},\eta\, D^h_k P_{,tt}\rangle-\langle 2\,\mu_{\rm e}\,\eta\,\sym(\nabla\dhk u-D^h_k P)+ \lambda_{\rm e}\,\eta\,\tr(\nabla\dhk u-D^h_k P)\id\notag\\[1ex]
&\quad \qquad  + 2\,\mu_{\rm c}\,\eta\,\skyw(\nabla\dhk u-D^h_k P)-2\,\mu_{\mathrm{micro}} \,\eta\,\sym \dhk P-\lambda_{\mathrm{micro}}\tr(\dhk P)\id,\eta\,\dhk P_{,t}\rangle\notag\\[1ex]
&\quad \quad \quad +\mu_{\rm micro} L_{\rm c}^2\langle \eta\,\Curl \dhk P, \eta\,\Curl \dhk P_{,t}\rangle \Big]\,\di x\,.\notag
\end{align}

It is worth to underline that the regularity \eqref{assu} of the solution $(u,P)$ implies that all integrals in \eqref{4.3} are well-defined. Let 
\begin{equation}
\eta\,\dhk\sigma=2\,\mu_{\rm e}\,\eta\,\sym(\nabla\dhk u-D^h_k P)+ \lambda_{\rm e}\,\eta\,\tr(\nabla\dhk u-D^h_k P)\,\id+2\,\mu_{\rm c}\,\eta\,\skyw(\nabla\dhk u-D^h_k P)\,,
\label{4.4}
\end{equation}
where
\begin{align}
\sigma=&\,2\,\mu_{\rm e}\sym(\nabla u- P)+ \lambda_{\rm e}\tr(\nabla u- P)\,\id+2\,\mu_{\rm c}\skyw(\nabla u- P)
\end{align}
is the  (non-symmetric) Cauchy-stress  tensor.
Then,
\begin{align}\label{4.5}
\frac{d}{dt}\big(\E(\eta\, D^h_k u, \eta\, D^h_k P)(t)\big)=&\int_{\Omega}\Big[\langle \eta\, D^h_k u_{,t},\eta\, D^h_k u_{,tt}\rangle+\langle\eta\,\dhk\sigma,\eta\,\nabla\dhk u_{,t}\rangle\Big]\,\di x \\[1ex]
&+\int_{\Omega}\Big[\langle \eta\, D^h_k P_{,t},\eta\, D^h_k P_{,tt}\rangle-\langle\eta\,\dhk\sigma-2\,\mu_{\mathrm{micro}} \,\eta\,\sym \dhk P-\lambda_{\mathrm{micro}}\tr(\dhk P)\id,\eta\,\dhk P_{,t}\rangle\notag\\[1ex]
&\qquad \quad +\mu_{\rm micro} L_{\rm c}^2\langle\eta\,\Curl \dhk P, \eta\,\Curl \dhk P_{,t}\rangle \Big]\,\di x\,.\notag
\end{align}
Denote by $(\cdot)_i$ ($i=1,2,3$) the rows of a $3\times 3$ matrix and set $u=(u_1,u_2,u_3)$. Therefore, for $i=1,2,3$ we conclude that 
\begin{equation}
\begin{split}
\dyw\big(\eta^2\dhk u_{i,t}\,\dhk\sigma_i\big)&= \langle\dhk\sigma_i,\eta^2\nabla\dhk u_{i,t}+2\,\eta\,\nabla\eta\,\dhk u_{i,t}\rangle+\eta^2\dhk u_{i,t}\dyw(\dhk\sigma_i)\,
\end{split}
\label{4.6}
\end{equation}
and 
\begin{equation}
\begin{split}
\sum_{i=1}^{3} \langle\eta\, \dhk\sigma_i,\eta\,\nabla\dhk u_{i,t}\rangle=&\sum_{i=1}^{3} \dyw\big(\eta^2\dhk u_{i,t}\,\dhk\sigma_i\big)-\sum_{i=1}^{3} \langle\eta\,\dhk\sigma_i,2\nabla\eta\,\dhk u_{i,t}\rangle-\sum_{i=1}^{3} \eta^2\dhk u_{i,t}\dyw(\dhk\sigma_i)\,.
\end{split}
\label{4.7}
\end{equation}
Moreover, we have
\begin{align}\label{4.8}
\dyw\big(\eta^2\,\mu_{\rm micro} L_{\rm c}^2\curl\dhk P_{i}\times\dhk P_{i,t}&\big)\!\!= \langle\dhk P_{i,t},\curl\big(\eta^2\,\mu_{\rm micro} L_{\rm c}^2\curl\dhk P_{i}\big)\rangle\notag\\[1ex]
&\qquad \qquad -\langle\eta^2\,\mu_{\rm micro} L_{\rm c}^2\curl\dhk P_{i},\curl\dhk P_{i,t}\rangle\\[1ex]
&=
\langle\dhk P_{i,t},2\eta\,\nabla\eta\,\times\mu_{\rm micro} L_{\rm c}^2\curl\dhk P_{i}+\eta^2\curl\big(\mu_{\rm micro} L_{\rm c}^2\curl\dhk P_{i}\big)\rangle\notag\\[1ex]
&\qquad \qquad-\langle \eta^2\,\mu_{\rm micro} L_{\rm c}^2\curl\dhk P_{i},\curl\dhk P_{i,t}\rangle\,.\notag
\end{align}
This leads to
\begin{equation}
\begin{split}
\sum_{i=1}^{3}& \mu_{\rm micro} L_{\rm c}^2\langle\eta\,\curl\dhk P_{i},\eta\,\curl\dhk P_{i,t}\rangle=-\sum_{i=1}^{3} \dyw\big(\eta^2\,\mu_{\rm micro} L_{\rm c}^2\curl\dhk P_{i}\times \dhk P_{i,t}\big)\\[1ex]
&+\sum_{i=1}^{3}\langle\eta\, \dhk P_{i,t},\eta\,\curl\big(\mu_{\rm micro} L_{\rm c}^2\curl\dhk P_{i}\big)\rangle+\sum_{i=1}^{3} \langle\eta\, \dhk P_{i,t},2\,\nabla\eta\,\times\mu_{\rm micro} L_{\rm c}^2\curl\dhk P_{i}\rangle\,.
\end{split}
\label{4.9}
\end{equation}
Inserting \eqref{4.7} and \eqref{4.9} into \eqref{4.5} and using equations \eqref{1.2} we arrive at
\begin{align}\label{4.10}
\frac{d}{dt}\big(\E(\eta\, D^h_k u, \eta\, D^h_k P)(t)\big)=&\int_{\Omega}\langle\eta\, D^h_k u_{,t},\eta\, \dhk f)\rangle\,\di x -2\int_{\Omega}\sum_{i=1}^{3} \langle\eta\,\dhk\sigma_i,\nabla\eta\,\dhk u_{i,t}\rangle,\di x\\[1ex]
&+\int_{\Omega}\langle\eta\, D^h_k P_{,t},\eta\, D^h_k M\rangle\,\di x+2\int_{\Omega}\sum_{i=1}^{3} \langle\eta\, \dhk P_{i,t},\nabla\eta\,\times\mu_{\rm micro} L_{\rm c}^2\curl\dhk P_{i}\rangle\,\di x\,.\notag
\end{align}
Notice that the divergence theorem and the properties \eqref{4.1} of the function $\eta$ show that the integrals over $\Omega$ of the first terms on the right-hand side of \eqref{4.7} and \eqref{4.9} are equal to zero. Now integrating with respect to time we obtain
\begin{equation}
\begin{split}
\E(\eta\, D^h_k u,& \eta\, D^h_k P)(t)=\E(\eta\, D^h_k u, \eta\, D^h_k P)(0)+\int_0^t\int_{\Omega}\langle\eta\, D^h_k u_{,t},\eta\, \dhk f)\rangle\,\di x\di \tau\\[1ex] 
&-2\int_0^t\int_{\Omega}\sum_{i=1}^{3} \langle\eta\,\dhk\sigma_i,\nabla\eta\,\dhk u_{i,t}\rangle\,\di x\di\tau+\int_0^t\int_{\Omega}\langle\eta\, D^h_k P_{,t},\eta\, D^h_k M\rangle\,\di x\di\tau\\[1ex]
&-2\int_0^t\int_{\Omega}\sum_{i=1}^{3} \langle\eta\, \dhk P_{i,t},\nabla\eta\,\times\mu_{\rm micro} L_{\rm c}^2\curl\dhk P_{i}\rangle\,\di x\di\tau\,.
\end{split} 
\label{4.11}
\end{equation}
The first integral on the right-hand side of \eqref{4.11} is estimated as follows
\begin{equation}
\begin{split}
\int_0^t\int_{\Omega}(\eta\, D^h_k u_{,t},\eta\, \dhk f)\,\di x\di \tau\leq \int_0^t  \lVert \nabla u_{,t} \rVert_ {{\rm L}^2(\Omega)} \lVert \nabla f \rVert_ {{\rm L}^2(\mathrm{supp\,\eta})}\,\di\tau
\end{split} 
\label{4.12}
\end{equation}
and the assumption on $f$ and the regularity of $u_{,t}$ yield that it is finite. Young's inequality implies that the second integral on the right-hand side of \eqref{4.11} can be estimated as follows
\begin{equation}
\begin{split}
\int_0^t\int_{\Omega}\sum_{i=1}^{3} \langle\eta\,\dhk\sigma_i,\nabla\eta\,\dhk u_{i,t}\rangle\,\di x\di\tau&\leq \int_0^t\int_{\Omega} \eta^2  \lVert \dhk\sigma\rVert^2\,\di x\di \tau +C \lVert \nabla\eta\,\rVert^2_{L^{\infty}(\Omega)}\int_0^t  \lVert \nabla u_{,t} \rVert_ {{\rm L}^2(\Omega)}^2\,\di\tau\\[1ex]
&\leq \int_0^t\E(\eta\, D^h_k u, \eta\, D^h_k P)(\tau)\,\di\tau+C \lVert \nabla\eta\,\rVert^2_{L^{\infty}(\Omega)}\int_0^t  \lVert \nabla u_{,t} \rVert_ {{\rm L}^2(\Omega)}^2\,\di\tau\,.
\end{split} 
\label{4.13}
\end{equation}
Again, using the regularity of $u_{,t}$ we have that the second integral on the right-hand side of \eqref{4.13} is finite. In turn, the third integral of the right-hand side of \eqref{4.11} is evaluated as
\begin{equation}
\begin{split}
\int_0^t\int_{\Omega}\langle\eta\, D^h_k P_{,t},\eta\, D^h_k M)\rangle\,\di x\di\tau&\leq \int_0^t\int_{\Omega} \eta^2  \lVert \dhk P_{,t}\rVert^2\,\di x\di\tau +C\int_0^t\int_{\Omega} \lVert \nabla M\rVert^2\,\di x\di \tau\\[1ex]
&\leq \int_0^t\E(\eta\, D^h_k u, \eta\, D^h_k P)(\tau)\,\di\tau+C\,.
\end{split} 
\label{4.14}
\end{equation}
Additionally, since
$
\curl\dhk P_{i}=\dhk(\curl P_i)\,,
$
we obtain 
\begin{equation}
\begin{split}
\int_0^t\int_{\Omega}&\sum_{i=1}^{3} \langle\eta\, \dhk P_{i,t},\nabla\eta\,\times\mu_{\rm micro} L_{\rm c}^2\curl\dhk P_{i}\rangle\,\di x\di\tau\\[1ex]
&\leq \int_0^t\int_{\Omega} \eta^2  \lVert \dhk P_{,t}\rVert^2\,\di x\di \tau+ C \lVert \nabla\eta\,\rVert^2_{L^{\infty}(\Omega)}\int_0^t  \lVert \nabla\Curl P \rVert_ {{\rm L}^2(\Omega)}^2\,\di\tau\\[1ex]
&\leq \int_0^t\E(\eta\, D^h_k u, \eta\, D^h_k P)(\tau)\,\di\tau+C \lVert \nabla\eta\,\rVert^2_{L^{\infty}(\Omega)}\int_0^t  \lVert \nabla\Curl P \rVert_ {{\rm L}^2(\Omega)}^2\,\di\tau\,.
\end{split} 
\label{4.16}
\end{equation}
Inequality \eqref{ineq} yields that the second term on the right-hand side of \eqref{4.16} is bounded. Substituting \eqref{4.12}-\eqref{4.16} into \eqref{4.11} we get
\begin{equation}
\begin{split}
\E&(\eta\, D^h_k u, \eta\, D^h_k P)(t)\leq\E(\eta\, D^h_k u, \eta\, D^h_k P)(0)+C\big(\int_0^t\E(\eta\, D^h_k u, \eta\, D^h_k P)(\tau)\,\di\tau+1\big)\,.\\[1ex] 
\end{split} 
\label{4.17}
\end{equation}
Thus, from Gronwall's inequality 
\begin{equation}
\begin{split}
\E&(\eta\, D^h_k u, \eta\, D^h_k P)(t)\leq C\big(\E(\eta\, D^h_k u, \eta\, D^h_k P)(0)+1\big)\,.\\[1ex] 
\end{split} 
\label{4.18}
\end{equation}
for all $t\in (0,T)$. Applying the regularity of the initial data we obtain
$
\E(\eta\, D^h_k u, \eta\, D^h_k P)(t)\leq C\,, 
$
where the constant $C>0$ depends on the length of time interval $(0,T)$ ($C$ does not depend on $h$).
\end{proof}
\subsection{The main result}
\begin{tw}	\label{regularity} \label{mr} {\rm (Regularity of the solution)} Suppose that all hypotheses of Theorem \ref{tw:regu} hold. Moreover, let $P^{(0)}\in {\rm H}^1_{\rm loc}(\Omega)$. Then, 
	\begin{equation}
	u\in {\rm L}^{\infty}(0,T;{\rm H}^2_{\rm loc}(\Omega))\,,\quad P\in {\rm L}^{\infty}(0,T;{\rm H}^1_{\rm loc}(\Omega))\quad \mathrm{and}\quad 
	\Curl P\in {\rm L}^{\infty}(0,T;{\rm H}^1(\Omega))\,.
	\label{regu1}
	\end{equation}
\end{tw}
\begin{proof} The proof is divided into two parts. First we are going to prove that $P\in {\rm L}^{\infty}(0,T;{\rm H}^1_{\rm loc}(\Omega))$ and $\Curl P\in {\rm L}^{\infty}(0,T;{\rm H}^1(\Omega))$.
\begin{itemize}
	\item Theorem \ref{tw:regu} implies that $ \lVert \dhk P_{,t}\rVert^2_{{\rm L}^{\infty}(0,T;{\rm L}^2(V))}\leq C$, thus from Theorem \ref{tw:difquo} we deduce 
	\begin{equation}
	\begin{split}
	 \lVert \nabla P_{,t}\rVert^2_{{\rm L}^{\infty}(0,T;{\rm L}^2(V))}\leq C\,.
	\end{split} 
	\label{4.20}
	\end{equation}
	Moreover, using the formula
	$
	\nabla P(t)=\nabla P^{(0)}+\int_0^t \nabla P_{,\tau}(\tau) \,\di\tau
	$
	we get 
	\begin{equation}
	\begin{split}
	 \lVert \nabla P(t) \rVert_ {{\rm L}^2(V)}= \lVert \nabla P^{(0)} \rVert_ {{\rm L}^2(V)}+\int_0^T \lVert \nabla P_{,\tau}(\tau) \rVert_ {{\rm L}^2(V)} \,\di\tau\,.
	\end{split} 
	\label{4.22}
	\end{equation}
	The regularity of $P^{(0)}$ and inequality \eqref{4.20} yield that 
	$
	 \lVert \nabla P\rVert^2_{{\rm L}^{\infty}(0,T;{\rm L}^2(V))}\leq C
	$
	and $P\in {\rm L}^{\infty}(0,T;{\rm H}^1_{\rm loc}(V))$. Notice that inequality \eqref{ineq} gives us also that $\Curl P\in {\rm L}^{\infty}(0,T;{\rm H}^1(\Omega))$.
	\item Observe that 
	\begin{equation}
	\begin{split}
	 \lVert \eta\,\sym(\nabla\dhk(u))\rVert^2_{{\rm L}^2(\Omega)}&\leq  2\Big(\lVert \eta\,\sym(\nabla\dhk(u)-\dhk P)\rVert^2_{{\rm L}^2(\Omega)}+ \lVert \eta\,\sym(\dhk P)\rVert^2_{{\rm L}^2(\Omega)}\Big)
	\end{split} 
	\label{4.24}
	\end{equation}
	and Theorem \ref{tw:regu} implies that $ \lVert \eta\,\sym(\nabla\dhk(u))\rVert^2_{{\rm L}^2(\Omega)}$ is bounded independently on $h$. Now, the regularity of $u$ follows from the identities
	\begin{equation}
	\begin{split}
	\eta\,(\partial_k\nabla u)&=\nabla(\eta\,\partial_k u)-\nabla\eta\,\otimes\partial_k u\,,\\[1ex]
	\sym(\nabla(\eta\,\partial_k u))&=\eta\,(\partial_k\sym(\nabla u))+\sym(\nabla\eta\,\otimes\partial_k u)
	\end{split} 
	\label{4.25}
	\end{equation}
	and from Korn's inequality \cite{Neff00b},  in the following form
	\begin{equation}
	\begin{split}
	 \lVert \eta\,(\partial_k\nabla u)\rVert^2_{{\rm L}^2(\Omega)}&\leq  2\Big(\lVert \nabla(\eta\,\partial_k u)\rVert^2_{{\rm L}^2(\Omega)}
	+ \lVert \nabla\eta\,\otimes\partial_k u\rVert^2_{{\rm L}^2(\Omega)}\Big)\\[1ex]
	&\leq C \Big(\lVert \sym(\nabla(\eta\,\partial_k u))\rVert^2_{{\rm L}^2(\Omega)} +  \lVert \nabla\eta\, \rVert_ {L^{\infty}(\Omega)} \lVert \nabla u\rVert^2_{{\rm L}^2(\Omega)}\Big)\\[1ex]
	&\leq C\Big( \lVert \eta\,(\partial_k\sym(\nabla u))\rVert^2_{{\rm L}^2(\Omega)}+ \lVert \nabla u\rVert^2_{{\rm L}^2(\Omega)}\Big)\,.
	\end{split} 
	\label{4.26}
	\end{equation}
	This proves that $u\in {\rm L}^{\infty}(0,T;{\rm H}^2_{\rm loc}(\Omega))$. \qedhere
\end{itemize}
\end{proof}
\begin{remark} Note that the regularity for the displacement vector $u$ is obtained from the isotropic elastic energy and microstrain self energy. Comparing this with the general regularity theory for hyperbolic equations, it is a standard approach \cite{neff2008regularity}. It is clear that locally the dislocation energy $\lVert \Curl P\rVert^2$ does not control all weak derivatives of the tensor $P$ in ${\rm L}^2$. But, in the dynamic case the total energy contains also the kinematic energy. From the difference quotient method and the energy estimate \eqref{estimate} we are able to control all weak partial  derivatives of the micro-distortion tensor $P_{,t}$ locally in ${\rm L}^2$. Assuming that the initial tensor $P^{(0)}$ has a better regularity, we also control all weak partial derivatives of the tensor $P$ locally in ${\rm L}^2$. 
\end{remark}
\bibliographystyle{plain} 
\begin{footnotesize}
	\noindent{\bf Acknowledgements:}  	The  work of I.D. Ghiba  was supported by a grant of the Romanian Ministry of Research
	and Innovation, CNCS--UEFISCDI, Project no.
	PN-III-P1-1.1-TE-2019-0348, Contract No. TE 8/2020, within PNCDI III.
	
	\medskip
	
	\noindent \textbf{Conflicts of interest:}\ This work does not have
	any conflicts of interest.

\end{footnotesize}
\end{document}